\documentclass{article}
\usepackage{amsmath, amssymb, amscd, amsthm}
\usepackage{pgf}
\usepackage{tikz}
\title{Projective homogeneous varieties birational to quadrics}
\author{Mark MacDonald}

\begin{document}

\newtheorem{thm}{Theorem}[section]
\newtheorem{lem}[thm]{Lemma}
\newtheorem{prop}[thm]{Proposition}
\newtheorem{cor}[thm]{Corollary}
\newtheorem{claim}[thm]{Claim}
\theoremstyle{definition}
\newtheorem{defn}[thm]{Definition}
\theoremstyle{remark}
\newtheorem{remark}[thm]{\textbf{Remark}}
\newtheorem{ex}[thm]{\textbf{Example}}

\newtheoremstyle{blank}{2mm}{0mm}{\normalsize}{0mm}{\textbf}{ }{0mm}{\thethm}
\theoremstyle{blank}
\newtheorem{blank}[thm]{}

\renewcommand{\thesubsubsection}{\thethm}

\maketitle

\begin{abstract}
We will consider an explicit birational map between a quadric and
the projective variety $X(J)$ of traceless rank one elements in a
simple reduced Jordan algebra $J$. $X(J)$ is a homogeneous $G$-variety for the automorphism group
$G=\textup{Aut}(J)$. We will show that the birational map is a
blow up followed by a blow down. This will allow us to use the
blow up formula for motives together with Vishik's work on the
motives of quadrics to give a motivic decomposition of $X(J)$.
\end{abstract}

Recently Totaro has solved the birational classification problem for a large class of quadrics \cite{To08}. In particular, let $\phi$ be an $r$-Pfister form over a field $k$ of characteristic not 2, and $b=\langle b_1,\cdots b_n \rangle$ be a non-degenerate quadratic form with $n\geq 2$. \begin{prop}\label{TotaroBirat}\cite[Thm.\ 6.3]{To08} The birational class of the quadric defined by $$q=\phi \otimes \langle b_1, \cdots, b_{n-1} \rangle \perp \langle b_n \rangle$$ only depends on the isometry classes of $\phi$ and $\phi \otimes b$, and \textit{not} on the choice of diagonalization of $b$.\end{prop}

The Sarkisov program \cite{Co94} predicts that any birational map between quadrics (in fact between any two Mori fibre spaces) factors as a chain of composites of ``elementary links''. In \ref{ChainSect} we will explicitly factor many of Totaro's birational maps into chains of elementary links, and also prove the following theorem.
\begin{thm}\label{SarkisovThm} For $r=0,1,2$ and $n\geq 3$, or $r=3$ and $n=3$, for each of the birational equivalences from Prop.\ \ref{TotaroBirat}, there is a birational map which factors into two elementary links, each of which is the blow up of a reduced subscheme followed by a blow down. Furthermore, if $r\neq 1$ or $\phi$ is not hyperbolic, then the intermediate Mori fibre space of this factorization will be the projective homogeneous variety $X(J)$ of traceless rank one elements in a Jordan algebra $J$. \end{thm}

The birational map
from a quadric to $X(J)$ will be the codimension 1 restriction of a
birational map between projective space and the projective variety
$V_J$ of rank one elements of $J$, first written down by Jacobson
\cite[4.26]{Ja85}.

\begin{blank} \textbf{Motivic decompositions.} Let $G$ a semisimple linear algebraic group of inner type, and $X$
a projective homogeneous $G$-variety such that $G$ splits over the
function field of $X$, which is to say, $X$ is generically split
(see \cite[3.6]{PSZ08} for a convenient table). Then \cite{PSZ08}
gives a direct sum decomposition of the Chow motive
$\mathcal{M}(X; \mathbb{Z}/p\mathbb{Z})$ of $X$. They show that it
is the direct sum of some Tate twists of a single indecomposable
motive $\mathcal{R}_p(G)$, which generalizes the Rost motive. This
work unified much of what was previously known about motivic
decompositions of anisotropic projective homogeneous varieties. 
\end{blank}

In the non-generically split cases less is known. Quadrics are in general not generically split, but much is known by the work of Vishik and others, especially in low dimensions \cite{Vi04}.
\begin{thm}(See Thm.\ \ref{Q(J,u)motive}) The motive of the projective quadric defined by the quadratic forms in Prop.\ \ref{TotaroBirat} may be decomposed into the sum, up to Tate twists, of Rost motives and higher forms of Rost motives. \end{thm} 
In the present paper we will use this knowledge of motives of quadrics to produce motivic decompositions for
the non-generically split projective homogeneous
$G$-varieties $X(J)$ which appear in Thm.\ \ref{SarkisovThm}. The algebraic groups $G$ are of Lie type $^2A_{n-1}$, $C_n$ and
$F_4$, and are automorphism groups of simple reduced Jordan algebras of
degree $\geq 3$. These varieties $X(J)$ come in four different types which we label $r=0,1,2$ or $3$, corresponding to the $2^r$ dimensional composition algebra of the simple Jordan algebra $J$ (see Thm.\ \ref{Aut(J)-orbits} for a description of $X(J)$ as $G/P$ for a parabolic subgroup $P$).
\begin{thm} (See Thm.\ \ref{MotiveX(J)}) The motive of $X(J)$ is the direct sum of a higher form of a Rost motive, $F^r_n$, together with several Tate twisted copies of the Rost motive $R^r$. \end{thm}
The $r=1$ case of this theorem provides an alternate proof of Krashen's motivic equivalence \cite[Thm.\ 3.3]{Kr07}. On the other hand, the $r=1$ case of this theorem is shown in \cite[Thm.\ (C)]{SZ08} by using Krashen's result (See Remark \ref{Krashen}).

\begin{blank} \textbf{Notational conventions.}
We will fix a base field $k$ of characteristic 0 (unless stated otherwise), and an algebraically closed (equivalently, a separably closed) field extension $\bar{k}$ of $k$. We only use the characteristic 0 assumption to show the varieties $X(J)$ and $Z_1$ are homogeneous. 
We will assume a \textit{scheme} over $k$ is a separated scheme of finite type over $k$, and a \textit{variety} will be an irreducible reduced scheme.\\
For a scheme $X$ over $k$, $\overline{X}=X \times_k \bar{k}$. \\
$G$ denotes an algebraic group over $k$. \\ $a_i$ are coefficients of the $r$-Pfister form $\phi$ over $k$.\\ $b_i$ are coefficients of the $n$-dimensional quadratic form $b$ over $k$.\\ $q$ denotes a quadratic form over $k$, and $Q$ is the associated projective quadric.\\ $i_W(q)$ is the Witt index of the quadratic form $q$. \\$C$ is a composition algebra (not to be confused with the Lie type $C_n$), and $c_i$ are elements of $C$.\\ $J$ is a Jordan algebra, $x$ is an element of $J$, and $u$ is an idempotent in $J$. \\
$X(J)$, $Q(J,u)$, $Z_1$ and $Z_2$ are complete schemes over $k$ defined in Section \ref{SarkisovSect}.\\
$F^r_n$ and $R^r$ are motives defined in Section \ref{SectMotivesNeighbours} (not to be confused with the Lie type $F_4$). \\
$\mathcal{M}(X)$ denotes the motive of a smooth complete scheme $X$, and $M\{i\}$ denotes the $i^{\textup{th}}$ Tate twist of the motive $M$.
\end{blank}
The paper is organized as follows. In Section \ref{JordanAlgSect} we will recall the terminology and classification of reduced simple Jordan algebras. In Section \ref{SarkisovSect} we describe the variety $X(J)$ and show it is homogeneous. Also we will define the birational map $v_2$ from a quadric to $X(J)$ and show that it is a Sarkisov link by analyzing its scheme of base points. In Section \ref{MotiveSect} we deduce motivic decompositions for a class of quadrics, as well as for the indeterminacy locus of $v_2$ introduced in Section \ref{SarkisovSect}. Finally we put these decompositions together to give a motivic decomposition of $X(J)$.

\section{Jordan algebras} \label{JordanAlgSect}

A \textit{Jordan algebra} over $k$ is a commutative, unital (not necessarily associative) $k$-algebra $J$ whose elements obey the identity $$x^2(xy)=x(x^2y) \text{  for all }x,y\in J.$$ A \textit{simple} Jordan algebra is one with no proper ideals. An \textit{idempotent} in $J$ is an element $u^2=u\neq 0 \in J$. Two idempotents are \textit{orthogonal} if they multiply to zero, and an idempotent is \textit{primitive} if it is not the sum of two orthogonal idempotents in $J$. For any field extension $l/k$, we can \textit{extend scalars} to $l$ by taking $J_l = J \otimes_k l$, for example $\bar{J}=J \otimes \bar{k}$. A Jordan algebra has \textit{degree} $n$ if the identity in $\bar{J}$ decomposes into $n$ pairwise orthogonal primitive idempotents over $\bar{k}$. A degree $n$ Jordan algebra is \textit{reduced} if the identity decomposes into $n$ orthogonal primitive idempotents over $k$.

The classification of reduced simple Jordan algebras of degree $\geq 3$ is closely related to the classification of composition algebras.
A \textit{composition algebra} over $k$ is a unital $k$-algebra $C$ together with a non-degenerate quadratic form $\phi$ on $C$ (called the \textit{norm form}) such that for any $c_1,c_2\in C$ we have that $\phi(c_1c_2)=\phi(c_1)\phi(c_2)$.
Two composition algebras are isomorphic as $k$-algebra iff their norm forms are isometric.
Every norm form is an $r$-fold \textit{Pfister form}, which is to say $$\phi=\langle \langle a_1, \cdots, a_r \rangle \rangle := \langle 1, -a_1 \rangle \otimes \cdots \otimes \langle 1 ,-a_r \rangle.$$
Furthermore, $r$ must be $0,1,2$ or $3$, and for any such $r$-fold Pfister form $\phi$, there is a composition algebra with $\phi$ as its norm form and a natural conjugation map $^-:C\to C$.

Let $C$ be a composition algebra with norm form $\phi = \langle \langle a_1, \cdots , a_r \rangle \rangle$, and let $b=\langle b_1, \cdots , b_n \rangle$ be a non-degenerate quadratic form.
Then we can define a reduced Jordan algebra in the following way.
Let $\Gamma=\textup{diag}(b_1,\cdots ,b_n)$, and let $\sigma_b(x):=\Gamma^{-1}\bar{x}^t \Gamma$ define a map from $M_n(C)$ to $M_n(C)$.
Then $\sigma_b$ is an involution (i.e.\ an anti-homomorphism such that $\sigma_b^2=\sigma_b$), so we can define $\textup{Sym}(M_n(C),\sigma_b)$ to be the commutative algebra of symmetric elements (i.e. elements $x$ such that $\sigma_b(x)=x$).
The product structure is defined by $x \circ y = \frac{1}{2}(xy+yx)$, using the multiplication in $C$.
 When $C$ is associative (i.e. $r=0,1$ or $2$) we know $\textup{Sym}(M_n(C),\sigma_b)$ is Jordan. For $r=3$, it is only Jordan when $n\leq 3$, so in what follows we will always impose this condition in the $r=3$ case.

The Jordan algebra isomorphism class of $\textup{Sym}(M_n(C),\sigma_b)$ only depends on the isomorphism classes of $b$ and $C$, and not on the diagonalization we have chosen for $b$. The following theorem states that in degrees $\geq 3$ these make up all of the reduced Jordan algebras up to isomorphism.

\begin{thm}(\textbf{Coordinatization} \cite[17]{Mc04},\cite[p.137]{Ja68})
Let $J$ be a reduced simple Jordan algebra of degree $n\geq 3$. Then there exists a composition algebra $C$ and an $n$-dimensional  quadratic form $b$ such that $J \cong \textup{Sym}(M_n(C),\sigma_b)$.
\end{thm}

\section{The Sarkisov link} \label{SarkisovSect}

We will define a birational map from a projective quadric to a projective homogeneous variety, $X(J)$, and show it is an elementary link in terms of Sarkisov (see \ref{SarkisovLink}).

Let $r=0,1,2,3$ and $n\geq 3$, and if $r=3$ then $n=3$. Throughout we will fix a composition
algebra $C$ of dimension $2^r$ over $k$, and elements $b_i\in k^*$ such that 
$b=\langle b_1,\cdots,b_n\rangle$ is a non-degenerate quadratic form. Let $J=
\textup{Sym}(M_n(C),\sigma_b)$ (see Section \ref{JordanAlgSect}). Then $J$ is a
central simple reduced Jordan algebra. Jacobson defined the closed subset $V_J \subset \mathbb{P}J$ of \textit{rank 1} elements of $J$ (he used the terminology \textit{reduced elements}) and showed it is a variety defined over $k$ \cite[\S 4]{Ja85}.

\begin{blank} \textbf{The Veronese map.} The following
rational map is a generalization of the $r=0$ case where it is the degree 2 Veronese morphism \cite[3]{Ch06} \cite[Last page]{Za93}.
\begin{align*} v_2:\mathbb{P}(C^n) &\dasharrow \mathbb{P}J \\
[c_1,\cdots,c_n] &\mapsto [b_ic_i\bar{c_j}].\end{align*} If the
composition algebra is associative (so $r\neq 3$), then the set-theoretic image
of $v_2$ (where it is defined) is precisely $V_J$. If $r=3$, then the set-theoretic image of $v_2$ isn't closed, but its closure is $V_J$ \cite[Prop.\ 4.2]{Ch06}. Note that this map depends on the choice of $n$ orthogonal primitive idempotents, $v_2([0,\cdots,1,\cdots,0])$, so it depends on more than just the isomorphism class of $J$.
\end{blank}

Let us restrict the map $v_2$ to the projective space
defined by $c_n\in k1$, and abuse notation by sometimes considering $v_2$ as a rational map from $\mathbb{P}(C^{n-1}\times k) \dasharrow V_J$. This map is an isomorphism on the open subset $U=(c_n\neq
0) \subset \mathbb{P}(C^{n-1}\times k)$ \cite[Thm.\ 4.26]{Ja85}, and hence birational. The
projective homogeneous variety we will be interested in is
$X(J)\subset V_J$ the hyperplane of traceless matrices, which has dimension $2^r(n-1)-1$.

\begin{blank} \label{DefnQ(J,u)}\textbf{The quadric $Q(J,u)$.} Define the quadric $Q(J,u)\subset
\mathbb{P}(C^{n-1} \times k)$ by $$\phi \otimes \langle
b_1,\cdots b_{n-1} \rangle \perp \langle b_n \rangle =
(\sum_{i=1}^{n-1}b_ic_i\bar{c_i}) + b_nc_n^2=0.$$ Here $\phi$ is the
norm form of $C$. The right hand side is simply the trace in
$V_J$, so the restriction of the birational map $v_2$ to $Q(J,u)$ has image in $X(J)$. We will often further abuse notation and consider $v_2$ to be the birational map from $Q(J,u)$ to $X(J)$.
\end{blank}

Although the definition of $Q(J,u)$ depends on the diagonalization of $b$, the isomorphism class of $Q(J,u)$ depends only on the isomorphism class of $J$ together with a choice of primitive idempotent $u$,
which we will usually take to be $u=\textup{diag}(0,\cdots,0,1)\in J$, as we have done above.

\begin{remark} \label{RemarkAltTotaro}
Since the birational class of $Q(J,u)$ is independent of $u\in J$, we have another proof of Prop.\ \ref{TotaroBirat} when $r\leq 3$, and if $r=3$ then $n=3$. For more on this, see \ref{ChainSect}.
\end{remark}

For connected algebraic groups $G$ over $\bar{k}$, projective homogeneous $G$-varieties $G/P$ are classified by conjugacy classes of parabolic subgroups $P$ in $G$. Furthermore, the conjugacy classes of parabolics are classified by specifying subsets $\theta$ of the set $\Delta$ of nodes of the Dynkin diagram of $G$, as in \cite[1.6]{Ti65}. In fact we will use the compliment to his notation, so that $\theta = \Delta$ corresponds to a Borel subgroup $P_{\Delta}=B$, and $\theta=\emptyset$ corresponds to $P_{\emptyset}=G$. We use the Bourbaki root numberings. $G^0$ denotes the connected component of the identity in $G$.

\begin{thm} \label{Aut(J)-orbits}
$V_J$ is the union of two $\textup{Aut}(J)$-orbits: $X(J)$ and
$V_J-X(J)$. Furthermore, we have: \\
(r=0): $\overline{X(J)} \cong G/P_{\theta}$, for $G=\textup{Aut}(\bar{J}) \cong \textup{SO}(n)$, if $n\neq 4$ then $\theta=\{1\}$, and if $n=4$ then the Dynkin diagram is two disjoint nodes, where $\theta$ is both nodes. In all cases, these varieties are quadrics. \\
(r=1): $\overline{X(J)} \cong G^0/P_{\theta}$, for $G=\textup{Aut}(\bar{J}) \cong \mathbb{Z}/2 \ltimes \textup{PGL}(n)$ and $\theta=\{1,n-1\}$, this is the variety of flags of dimension 1 and codimension 1 linear subspaces in a vector space.   \\
(r=2): $\overline{X(J)} \cong G/P_{\theta}$, for $G=\textup{Aut}(\bar{J}) \cong \textup{PSp}(2n)$ and $\theta=\{2\}$, this is the second symplectic Grassmannian. \\
(r=3): $\overline{X(J)} \cong G/P_{\theta}$, for $G=\textup{Aut}(\bar{J}) \cong F_4$ and $\theta=\{4\}$, this may be viewed as a hyperplane section of the Cayley plane.
\end{thm}

\begin{proof}
$\textup{Aut}(J)$ acts on $V_J$, since the rank is preserved by automorphisms. So it is sufficient to prove this theorem for $k=\bar{k}$.  Every element of $V_J-X(J)$ is $[u]$ for some rank one idempotent $u$ \cite[Prop.\ 3.8]{Ch06}, and
$\textup{Aut}(J)$ is transitive on rank one idempotents by
Jacobson's coordinatization theorem, since the field is
algebraically closed \cite[17]{Mc04}.

Clearly $X(J)$ is preserved by $\textup{Aut}(J)$, since the trace is
preserved by automorphisms.
All that remains is to show that $\textup{Aut}(J)$ is transitive on $X(J)$, which we will do in cases. Consider the $2^{r-1}n(n-1)+n$ dimensional $\textup{Aut}(J)$ representation $J=k\oplus J_0$, where $J_0$ is the subrepresentation of traceless elements in $J$.
In all cases we will show that $J_0$ is an irreducible $\textup{Aut}(J)$ representation, find the highest weight,
and show that there is a closed orbit in $\mathbb{P}(J_0)$ which is contained in $X(J)$ and is of the same dimension. Therefore, by uniqueness of the closed orbit, which follows from the irreducibility of $J_0$, $X(J)$ is the closed orbit.

\hspace{4mm} \textit{Case $r=0$:} For simplicity, we will modify the definition of $J$. Instead of taking $n
\times n$ matrices such that $x^t=x$, we will take matrices such that $M^{-1}x^tM=x$ where
$$M= \left[
  \begin{array}{ c c }
     0 & I_m \\
     I_m & 0
  \end{array} \right]
 \textup{ for $n=2m$, and  } M=\left[
  \begin{array}{ c c c }
     0 & I_m & 0 \\
     I_m & 0 & 0 \\
0 & 0 & 1
  \end{array} \right]
 \textup{ for $n=2m+1$.}$$
This change is justified by recalling that any two orthogonal involutions in the same matrix algebra over an algebraically closed field are isomorphic.
Now the Lie algebra of derivations $\textup{Der}(J)\cong \mathfrak{so}(n)$ is in the more standard form, and we can choose elements of the Cartan subalgebra $\mathfrak{h}$ as diagonal matrices $H_i=E_{i,i}-E_{m+i,m+i}$ as in \cite[18]{FH91}.
Following the conventions of \cite{FH91}, we have a dual basis $L_i(H_j)=\delta_{ij}$ of $\mathfrak{h}^*$, and we wish to find the highest weight of the representation $J_0$.

For $n=2m$, the roots of $\mathfrak{so}(2m)$ are $\pm L_i \pm L_j$ for $1\leq i\neq j \leq m$.
One can check that the non-zero weights of $J_0$ are $\pm L_i \pm L_j$ for all $i,j$.
In particular, the element $E_{1,m+1}$ is a weight vector in $J_0$ for the weight $2L_1$, and the irreducible representation with highest weight $2L_1$ is of the same dimension as $J_0$.
Therefore $J_0$ is the irreducible representation with highest weight $2L_1$, and since $\textup{Aut}(J)$ is simple, there is a unique closed orbit in $\mathbb{P}(J_0)$, and it is the orbit of $E_{1,m+1}$.
To determine the dimension of the orbit, we ask which root spaces $\mathfrak{g}_{-\alpha_i}$ in the Lie algebra for the negative simple roots $-\alpha_i$, kill the weight space of $2L_1$.
 For $n=4$, neither root space, for $-\alpha_1=-L_1-L_2$ nor $\alpha_2=-L_1+L_2$, kills this weight space.
For any $n\geq 6$ even, all of the negative simple root spaces kill the weight space $2L_1$ except for the one for $-L_1+L_2$.
 In either case the dimension of the parabolic fixing $E_{1,m+1}$ is $2m^2-3m+2$, so the dimension of the orbit is $n-2$.
This is the dimension of the closed invariant subset $X(J)$, which must contain a closed orbit.
Since there is only one closed orbit, $X(J)$ must be the entire orbit.

A similar analysis may be carried out in the $n=2m+1$ case, where again $E_{1,n+1}$ is a weight vector for the highest weight $2L_1$.

\hspace{4mm} \textit{Case $r=1$:} We have the action of the connected component $\textup{Aut}(J)^0=PGL(n)$ on $J\cong M_n(k)$, acting by conjugation.
The induced action of the Lie algebra of derivations $\textup{Der}(J)\cong \mathfrak{sl}(n)$ on $J_0$ is just the adjoint action on $\mathfrak{sl}(n)$.
With the standard diagonal Cartan subalgebra, and choice of positive roots dual to $H_i=E_{i,i}-E_{i+1,i+1}$, the highest weight is in the representation $J_0$ is $2L_1+L_2 +\cdots + L_{n-1}$ with multiplicity 1.
A dimension count shows this representation is irreducible, and the dimension of the parabolic fixing a highest weight vector is $n^2-2n+2$.
So the dimension of the unique closed orbit is $2n-3$, which is the dimension of $X(J)$. Therefore $X(J)$ is the closed orbit.

\hspace{4mm} \textit{Case $r=2$:} As in the $r=0$ case, we will change our symplectic involution $\sigma(x)=\bar{x}^t$ to $\sigma_M(x)=M^{-1} x^t M$ for $$M= \left[
  \begin{array}{ c c }
     0 & I_n \\
     -I_n & 0
  \end{array} \right].$$ Then the Lie algebra of derivations $\textup{Der}(J)\cong \mathfrak{sp}(2n)$ is in the standard form, by choosing a Cartan subalgebra of diagonal matrices, with $H_i=E_{i,i}-E_{n+i,n+i}$ and dual basis $L_i \in \mathfrak{h}^*$. The roots of $\mathfrak{sp}(2n)$ are $\pm L_i \pm L_j$ for all $i,j$, and the non-zero weights of $J_0$ are $\pm L_i \pm L_j$ for $i \neq j$. In particular, the highest weight is $L_1+L_2$ in the standard weight ordering of \cite[p.257]{FH91}. Comparing dimensions shows that $J_0$ is irreducible, and the parabolic fixing a highest weight vector is of dimension $2n^2-3n+5$. So the unique closed orbit in $\mathbb{P}(J_0)$ is of dimension $4n-5$, which is the same as the dimension of $X(J)$. Therefore $X(J)$ is the unique closed orbit.

\hspace{4mm} \textit{Case $r=3$:} First notice that $J_0$ is a 26-dimensional non-trivial representation of $F_4=\textup{Aut}(J)$. It is well-known that such a representation is unique, and has a 15-dimensional unique closed orbit in $\mathbb{P}(J_0)$. Since $X(J)$ is a 15-dimensional closed invariant subset, it must be equal to the closed orbit.
\end{proof}

\begin{remark}
Over the complex numbers the varieties with exactly two $G$-orbits
for some semisimple algebraic group $G$, one of which is of
codimension one, have been classified by \cite{Ah86}. The
varieties $V_J$ account for most of these.
\end{remark}

\stepcounter{thm}
\subsubsection{Blowing up the base loci} \label{BlowupSect}

Any birational map of projective varieties over a field can be expressed as a blow up followed a blow down of closed subschemes (Prop.\ \ref{Blowupscheme}).
In this section we will show that these closed subschemes, for our birational map from $Q(J,u)$ to $X(J)$, are (usually) smooth varieties, and hence see that the map is an elementary link in terms of Sarkisov.

Given a rational map between projective varieties $f:Y\dasharrow X$, we
can define the scheme of base points of $f$ as a closed subscheme
of $Y$ \cite[II. Example 7.17.3]{Ha77}. 

\begin{prop} \label{Blowupscheme}
Let $f:Y\dasharrow X$ be a birational map of projective varieties over a
field $k$ with $g:X\dasharrow Y$ the inverse birational map. Let $Z_Y$
and $Z_X$ be the schemes of base points of $f$ and $g$
respectively. Then the blow up $\tilde{Y}$ of $Y$ along $Z_Y$ is
isomorphic to the blow up $\tilde{X}$ of $X$ along $Z_X$.
\end{prop}

\begin{proof}
Let $U \subset Y$ be the open subset on which $f$ is an
isomorphism. Then the graph $\Gamma_f$ of $f|_U$ is a subset of $U
\times f(U) \subset Y \times X$. The closure of $\Gamma_f$ in $Y
\times X$, given the structure of a closed reduced subscheme, is the blow
up $\tilde{Y}$ \cite[Prop.\ IV.22]{EH00}\footnote{They assume $Y$
is affine, but we can drop this assumption since the blow up is
determined locally.}.

Similarly, $\tilde{X}$ is the closure of $\Gamma_g \subset U
\times f(U)$. Since the inverse of $f$ on $U$ is $g$, we have that
$\tilde{X}$ and $\tilde{Y}$ are both closures in $Y \times X$ of
the same subset of $U \times f(U)$. So they have the same
structure as reduced schemes, and hence $\tilde{X} \cong \tilde{Y}$.
\end{proof}

\begin{blank} \textbf{Indeterminacy locus of $v_2$.} Let $Z_1$ be the closed \textit{reduced} subscheme associated to the scheme of base points in $Q(J,u)$ of
the birational map $v_2$. We will show that $Z_1$ is isomorphic to the
scheme of base points. We denote by $\textup{Aut}(J,u)$ the subgroup of automorphisms of $J$ that fix the primitive idempotent $u$.\end{blank}

\begin{thm} \label{B1orbit}
$Z_1$ is homogeneous under an action of $\textup{Aut}(J,u)$.
\end{thm}

\begin{proof}
To describe the action we will use the vector space isomorphism $C^{n-1}\cong J_{\frac{1}{2}}(u)=\{x\in J | x\cdot u = \frac{1}{2}x\}$. Here, as above, we take $u=\textup{diag}(0, \cdots, 0,1)=E_{n,n}$. This isomorphism is given by sending an element $c\in C^{n-1}$ to the matrix element in $J_{\frac{1}{2}}(u)\subset M_n(C)$ with $n^{th}$ row equal to $[c,0]$.

So we have an $\textup{Aut}(J,u)$ action on $\mathbb{P}(C^{n-1})$. By considering the defining equations, one see that $Z_1$ is isomorphic to the reduced subscheme of $\mathbb{P}(J_{\frac{1}{2}}(u))$ defined by the matrix equation $x^2=0$. So it is clear that the underlying closed subset is stable under $\textup{Aut}(J,u)$.

Finally, to show the action is transitive, it is enough to show it after extending scalars to an algebraically closed field $\bar{k}$. We will use similar arguments as in the proof of Thm.\ \ref{Aut(J)-orbits}.

\hspace{4mm} \textit{Case $r=2$:} Using the notation from the proof of Thm.\ \ref{Aut(J)-orbits}, the roots of the Lie algebra of $\textup{Aut}(J,u)$ are $\pm L_i \pm L_j$ for $i,j \leq n-1$ together with $\pm 2L_n$.
One can check that the non-zero weights of the representation $J_{\frac{1}{2}}(u)$ are $\pm L_i \pm L_n$ for $i \leq n-1$.
A dimension count reveals that $J_{\frac{1}{2}}(u)$ is therefore an irreducible representation with highest weight $L_1+L_n$.
The only negative simple roots that don't kill a highest weight vector are $L_2-L_1$ and $-2L_n$,
so the dimension of the parabolic subgroup that fixes a point in the unique closed orbit in $\mathbb{P}(J_{\frac{1}{2}}(u))$ is $2n^2-5n+6$.
So the dimension of this orbit is $2n-2$.

To see this is the same as the dimension of $Z_1$,
consider the affine cone $\tilde{Z_1}$ over $Z_1$ inside $J_{\frac{1}{2}}(u)$.
Then consider Jacobian matrix of the equations given by $\{x_i\bar{x_j}=0\}$ with respect
to the $4(n-1)$ variables: 4 variables for each coordinate $x_i\in C$. The rank of this matrix at any
point in the affine cone over $Z_1$ is $\leq \textup{dim}(J_{\frac{1}{2}}(u))-\textup{dim}(\tilde{Z_1})$,
where equality holds if the ideal spanned by the polynomials $\{x_i\bar{x_j}\}$ is radical.
By choosing a convenient point, we see that the dimension of $Z_1$ is at most $2n-2$, which is the dimension of the closed orbit.
So if $Z_1$ contained another $\textup{Aut}(J,u)$-orbit, then it would contain another closed orbit.
But the closed orbit is unique, and therefore $Z_1$ is the closed orbit.

\hspace{4mm} \textit{Case $r=3$:} It is well known that the $\textup{Aut}(J,u)\cong\textup{Spin}(9)$ representation given by $J_{\frac{1}{2}}(u)$ for $u=E_{3,3}$ is the 16-dimensional spin representation. The unique closed orbit in $\mathbb{P}(J_{\frac{1}{2}}(u))$ is therefore the 10-dimensional spinor variety. Using a similar argument to the $r=2$ case, we can show the dimension of $Z_1$ is at most 10, so by the uniqueness of the closed orbit we can conclude that $Z_1$ is the closed orbit.

\hspace{4mm} \textit{Case $r=1$:} This case is slightly different from the other two because $\textup{Aut}(J,u) \cong
\mathbb{Z}/2 \ltimes GL(n-1)$ is a disconnected group, and the connected component has
\textit{two} closed orbits in $\mathbb{P}(J_{\frac{1}{2}}(u))$.
The argument is similar to the $r=2$ case, except we find that the $\mathfrak{sl}(n-1)$-representation $J_{\frac{1}{2}}(u)$ is the direct sum of the standard representation $V$ with its dual $V^*$.
So the two closed orbits in $\mathbb{P}(J_{\frac{1}{2}}(u))$ are the orbits of weight vectors for the weights $L_1-L_n$ and $L_n-L_1$, which are the respective closed orbits in $\mathbb{P}V$ and $\mathbb{P}V^*$.
Each $\mathfrak{sl}(n-1)$-orbit has dimension $n-2$.
Furthermore, the $\mathbb{Z}/2$ part of $\textup{Aut}(J,u)$ swaps these two representations, since it acts on matrices as the transpose.
So there is a unique closed $\textup{Aut}(J,u)$-orbit, and it is of dimension $n-2$.

As in the $r=2$ case, by considering the rank of the Jacobian at a closed point in $\tilde{Z_1}$, we see that the dimension of $Z_1$ is at most $n-2$. Since $Z_1$ is $\textup{Aut}(J,u)$-stable, we can conclude that it is the closed orbit.
\end{proof}

\begin{cor} \label{B1smooth}
The reduced scheme $Z_1$ is isomorphic to the scheme of base points of $v_2$ in $Q(J,u)$.
\end{cor}

\begin{proof}
The $r=0$ case is trivial, since $v_2$ is a morphism and hence $Z_1$
is empty. It is sufficient to assume $k$ is algebraically closed.

The other cases follow from the proof of Thm.\ \ref{B1orbit}, as follows.
We can choose a convenient  closed point in the scheme of base points,
and show that the rank of the Jacobian of the defining polynomials given by $\{v_2(x)=0\}$
is equal to the codimension. This implies the scheme is smooth at that point (and therefore at all points),
so in particular, it is reduced.
\end{proof}

\begin{cor}
Over $\bar{k}$, the smooth subscheme $Z_1$ is isomorphic to the following. \\
$(r=0): \emptyset$\\
$(r=1): \mathbb{P}^{n-2} \sqcup \mathbb{P}^{n-2}$ \\
$(r=2): \mathbb{P}^1 \times \mathbb{P}^{2n-3}$ \\
$(r=3):$ The 10-dimensional spinor variety
\end{cor}

\begin{proof}
This follows from our representation theoretic understanding of $Z_1$ from the proof of Thm.\ \ref{B1orbit}.

There are much more explicit ways of understanding the $r\neq 3$ cases. For example, in the $r=2$ case, if
$c=[c_1,\cdots,c_{n-1}]\in \mathbb{P}(M_2(\bar{k})^{n-1})$ is in
$\overline{Z_1}$, then the $c_i$'s are rank 1 matrices that have a common non-zero vector in their kernels.
This can be used to get an explicit isomorphism with $\mathbb{P}^1 \times \mathbb{P}^{2n-3}$.
\end{proof}

\begin{remark}
These varieties are written in \cite[Final pages]{Za93}, where it is implicitly suggested that they are the base locus of the rational map $v_2$.
\end{remark}

\begin{remark} \label{B1irreducible}
It is shown in \cite{Kr07} that $Z_1 \cong \textup{Spec}(k(\sqrt{a_1})) \times_k \mathbb{P}^{n-2}$, where $\langle \langle a_1 \rangle \rangle$ is the norm form associated to $C$. So the above corollary shows that $Z_1$ is irreducible over $k$ except for the single case when $r=1$ and $C$ is split.
\end{remark}

\begin{blank} \textbf{Indeterminacy locus of $v_2^{-1}$.} Let $Z_2$ be the scheme of base points
of the inverse birational map $v_2^{-1}:X(J_n) \dasharrow Q(J,u)$. We
have that $v_2^{-1}([x_{ij}])=[x_{n,1},\cdots,x_{n,n}]$, where this is defined.\end{blank}

We will use the notation $J_{n-1}=\textup{Sym}(M_{n-1}(C),
\sigma_{\langle b_1,\cdots,b_{n-1}\rangle})$, and sometimes
$J_n=J$ for emphasis. The isomorphism class of $J_{n-1}$ depends
on the choice of primitive idempotent $u=E_{n,n}\in J$, but is otherwise
independent of the diagonalization of $\langle
b_1,\cdots,b_{n-1}\rangle$.

\begin{lem} \label{B2smooth}
The scheme of base points $Z_2$ is isomorphic to the smooth subvariety
$X(J_{n-1})$.
\end{lem}

\begin{proof}
The indeterminacy locus of $v_2^{-1}$ is simply the closed subset of
matrices in $X(J_n)$ whose bottom row (and therefore right-most
column) is zero. In other words, $Z_2$ is defined by linear
polynomials. The ideal of these polynomials is radical, and therefore the scheme $Z_2$ is reduced.
For $n\geq 4$, one sees that $Z_2$ is isomorphic to $X(J_{n-1})$. For $n=3$, by considering the matrix equation $x^2=0$, we see that the base locus of $Z_2$ is the quadric defined by $\phi \otimes \langle b_1 \rangle \perp \langle b_2 \rangle =0$. We will define $X(J_2)$ to be this quadric.
\end{proof}

\stepcounter{thm}
\subsubsection{The chain between two quadrics} \label{ChainSect}

The Sarkisov program \cite{Co94} predicts that any birational map between two Mori fibre spaces $X$ and $Y$ factors into a chain of elementary links between intermediate Mori fibre spaces. An example of such a link (of type II \cite[3.4.2]{Co94}) would be $X \leftarrow W \rightarrow V$ where both morphisms are blow ups of smooth subvarieties, and $X$ and $V$ are projective homogeneous varieties with Picard number 1 (and hence Mori fibre spaces).

\begin{thm} \label{SarkisovLink}
For $r\neq 1$ or $C$ non-split, the birational map $v_2$ from $Q(J,u)$ to $X(J)$ is an elementary link of type II. 
\end{thm}

\begin{proof}
We have that $Z_1$ is irreducible (see Remark \ref{B1irreducible}). The blow up of an irreducible smooth subscheme increases the Picard number by 1, and a blow down decreases it by 1. So in this situation, by Lemma \ref{B2smooth} and Lemma \ref{B1smooth} we see that $X(J)$ has Picard number 1. So by Prop.\  \ref{Blowupscheme} we have that $v_2$ is a blow up of a smooth subvariety followed by a blow down to a smooth subvariety, and therefore it is an elementary link of type II. 
\end{proof}

Let $b'=\langle b'_1, \cdots, b'_n \rangle$, and $q'=\phi \otimes \langle b'_1, \cdots, b'_{n-1} \rangle \perp \langle b'_n \rangle$. Then Totaro's Prop.\ \ref{TotaroBirat} states that if $\phi \otimes b \cong \phi \otimes b'$, then the quadrics defined by $q$ and $q'$ are birational. By defining the Jordan algebra $J'$ using $\phi$ and $b'$, we have a birational map $v_2'$ from $Q(J',u')$ to $X(J')$.

\begin{proof}[Proof of Thm. \ref{SarkisovThm}] If $\phi \otimes b \cong \phi \otimes b'$, then the Jordan algebras $J \cong J'$ are isomorphic as algebras (\cite[Prop.\ 4.2, p.\ 43]{KMRT98}, \cite[Ch.\ V.7, p.\ 210]{Ja68}), and therefore the varieties $X(J) \cong X(J')$ are also isomorphic. So, as noted in Remark \ref{RemarkAltTotaro}, $Q(J,u)$ is birational to $Q(J',u')$, and moreover by Thm.\ \ref{SarkisovLink} this map is the composition of two elementary links, with intermediate variety $X(J)$. Notice that if $C$ is a split composition algebra (equivalently, $\phi$ is hyperbolic) then $Q(J,u)$ and $Q(J',u')$ are already isomorphic. \end{proof} 

\begin{blank} \textbf{Transposition maps.} Now we will explicitly factor the birational maps of Roussey and Totaro, which in general have more than two elementary links. The most basic case they consider, though, is that of \textit{transposition}. This corresponds to finding a birational map between quadrics $q$ and $q'$, where $b'_i=b_i$ for $1\leq i \leq n-2$, and $b'_{n-1} = b_n$, $b'_n=b_{n-1}$. So $b$ and $b'$ differ by transposing the last two entries. Totaro proves Prop.\ \ref{TotaroBirat} by finding a suitable chain of such transposition maps.
\end{blank}

\begin{prop} \label{TranspositionMaps}
For $r=0,1,2$ and $n\geq 3$, and if $r=3$ then $n=3$, Totaro's transposition map factors as the composite of two elementary links.
\end{prop} 

\begin{proof}
Let $q$ and $q'$ be as above, and let $J=\textup{Sym}(M_n(C_{\phi}), \sigma_b)$. Then the quadric $(q=0)=Q(J,u)$ is defined using the idempotent $u=\textup{diag}(0, \cdots 0,1)\in J$ (see \ref{DefnQ(J,u)}). General rational points on this quadric are elements in $\mathbb{P}(C^{n-1}\times k)$ such that $v_2([c_1, \cdots, c_n])\in \mathbb{P}J$ has trace zero. Here $c_i\in C$ for $i\neq n$, and $c_n \in k$. The inverse birational map $v_2^{-1}$ simply takes the $n^{th}$ row of the matrix in $J$.

Then the quadric for $(q'=0)=Q(J,u')$ can be defined using the idempotent $u'=\textup{diag}(0,\cdots,1,0)\in J$. General rational points on this quadric are elements in $\mathbb{P}(C^{n-2}\times k \times C)$ such that $v'_2([c'_1,\cdots, c'_n]) \in \mathbb{P}J$ has trace zero, where we use the same Jordan algebra $J$. Here $c'_i\in C$ for $i\neq n-1$, and $c'_{n-1} \in k$. The inverse birational map $(v'_2)^{-1}$ takes the $n-1^{th}$ row of the matrix in $J$.

So the composition $(v'_2)^{-1} \circ v_2$ defines a birational map from $Q(J,u)$ to $Q(J,u')$. From Thm.\ \ref{SarkisovLink} this is the composite of two elementary links. So it remains to show this composite is the same as Totaro's transposition map.

To see this, consider the map $(v'_2)^{-1} \circ v_2$ over $\bar{k}$, and observe where it sends a general point from $Q(J,u)$. Recall that $v_2$ sends $[c_1,\cdots,c_n]$ to the matrix $[b_ic_i\overline{c_j}]\in X(J)$, and then taking the $n-1^{th}$ row of this matrix gives us $$[b_{n-1}c_{n-1}\overline{c_1}, \cdots, b_{n-1}c_{n-1}\overline{c_{n-1}}, b_{n-1}c_{n-1}\overline{c_n}] \in Q(J,u') \subset \mathbb{P}(C^{n-2} \times \bar{k} \times C).$$ After using the isomorphism $\mathbb{P}(C^{n-2}\times k \times C) \cong \mathbb{P}(C^{n-1} \times k)$ to swap the last two coordinates, we can now recognize that this is exactly a map from \cite[Lemma 5.1]{To08}, where the ``multiplication'' of elements in $C$, is $x *y:=x\bar{y}$. 
\end{proof}

\begin{remark}
We may also view this chain of birational maps as a ``weak factorization" in the sense of \cite{AKMW02}. They prove that any birational map between smooth projective varieties can be factored into a sequence of blow ups and blow downs of smooth subvarieties. But a chain of Sarkisov links (of type II) is stronger, because then each blow up is immediately followed by a blow down, and the intermediate varieties are Mori fibre spaces.
\end{remark}

\section{Motives} \label{MotiveSect}

For a smooth complete scheme $X$ defined over $k$, we will denote the Chow motive of $X$ with coefficients in a ring $\Lambda$ by $\mathcal{M}(X;\Lambda)$, following \cite{EKM08} (see also \cite{Vi04}, \cite{Ma68}). We will briefly recall the definition of the category of graded Chow motives with coefficients in $\Lambda$. 

Let us define the category $\mathcal{C}(k,\Lambda)$. The objects will be pairs $(X,i)$ for $X$ a smooth complete scheme over $k$, and $i \in \mathbb{Z}$, and the morphisms will be \textit{correspondences}: $$\textup{Hom}_{\mathcal{C}(k,\Lambda)}((X,i),(Y,j)) = \bigsqcup_m \textup{CH}_{\textup{dim}(X_m)+i-j}(X_m \times_k Y, \Lambda).$$ Here $\{X_m\}$ is the set of irreducible components of $X$.
If $f:X \to Y$ is a morphism of $k$-schemes, then the graph of $f$ is an element of $\textup{Hom}_{\mathcal{C}(k,\Lambda)}((X,0),(Y,0))$. There is a natural composition on correspondences that generalizes the composition of morphisms of schemes.

We denote the \textit{additive completion} of this pre-additive category by $CR(k,\Lambda)$. Its objects are finite direct sums of objects in $\mathcal{C}(k,\Lambda)$, and the morphisms are matrices of morphisms in $\mathcal{C}(k,\Lambda)$. Then $CR(k,\Lambda)$ is the category of \textit{graded correspondences} over $k$ with coefficients in $\Lambda$. 

Finally, we let $CM(k,\Lambda)$ be the \textit{idempotent completion} of $CR(k,\Lambda)$. Here the objects are pairs $(A,e)$, where $A$ is an object in $CR(k,\Lambda)$ and $e \in \textup{Hom}_{CR(k,\Lambda)} (A,A)$ such that $e\circ e=e$. Then the morphisms are $$\textup{Hom}_{CM(k,\Lambda)}((A,e),(B,f))=f \circ \textup{Hom}_{CR(k,\Lambda)}(A,B) \circ e.$$ This is the category of \textup{graded Chow motives} over $k$ with coefficients in $\Lambda$. For any smooth complete scheme $X$ over $k$, we denote $\mathcal{M}(X)=((X,0),id_X)$ its Chow motive, and $\mathcal{M}(X)\{i\}=((X,i),id_X)$ its $i^{th}$ Tate twist. Any object in $CM(k,\Lambda)$ is the direct summand of a finite sum of motives $\mathcal{M}(X)\{i\}$.

In this section we will describe direct sum motivic decompositions of $Q(J,u)$, $Z_1$ and finally $X(J)$. A non-degenerate quadratic form $q$ of dimension $\geq 2$ defines a smooth projective quadric $Q$, and we will sometimes write $\mathcal{M}(q)=\mathcal{M}(Q)$.

\stepcounter{thm}
\subsubsection{Motives of neighbours of multiples of Pfister quadrics} \label{SectMotivesNeighbours}
In this section until Example \ref{MotiveExample} we can assume our base field $k$ is of any characteristic other than 2, and $r\geq 1$ may be arbitrarily large.
Given an $r$-fold Pfister form $\phi$ and an $n$-dimensional
non-degenerate quadratic form $b=\langle b_1, \cdots ,b_n \rangle$ over $k$
we will describe the motivic decomposition of the projective quadric $Q$
defined by $$q= \phi \otimes \langle b_1,\cdots,b_{n-1} \rangle
\perp \langle b_n \rangle.$$ This quadric is dependent on the
choice of diagonalization of $b$. We will use
Vishik's following motivic decomposition of the quadric defined by $\phi
\otimes b$.

\begin{thm} (\cite[6.1]{Vi04}) \label{Motiveq_h} \\
For $n\geq 1$, there exists a motive $F^r_n$ such
that
\[\mathcal{M}(\phi \otimes b) = \bigoplus_{i=0}^{2^r-1}F^r_n\{i\} \oplus \left\{
\begin{array}{l l}
  \emptyset & \quad \mbox{if $n$ is even}\\
  \mathcal{M}(\phi)\{2^{r-1}(n-1)\} & \quad \mbox{if $n$ is odd.}\\
\end{array} \right. \]
\end{thm}

Vishik uses the notation $F_{\phi}(\mathcal{M}(b))$ for $F^r_n$, and
calls it a \textit{higher form of} $\mathcal{M}(b)$. It only
depends on the isometry classes of $\phi$ and $b$. 

If $\phi$ is anisotropic, Rost defined an indecomposable motive $R^r$ such that $\mathcal{M}(\phi)$ is the direct sum of Tate twists of $R^r$. This is called the Rost motive of $\phi$. If $\phi$ is split, then this motive is no longer indecomposable, but we will still call $R^r = \mathbb{Z} \oplus \mathbb{Z}\{2^{r-1}-1\}$ the Rost motive. 
In fact, $F^r_2$ is just the Rost
motive of $\phi \otimes b$ (which is similar to a Pfister form).
Also note that $F^r_1=0$.

In particular, for $n\geq 1$, by counting Tate motives one sees that $$F^r_n|_{\bar{k}} = \bigoplus_{i=0}^{\lfloor
\frac{n}{2} \rfloor-1} (\mathbb{Z}\{2^ri\}\oplus
\mathbb{Z}\{2^r(n-1)-2^ri-1\}).$$ So the summand has $2\lfloor
\frac{n}{2} \rfloor$ Tate motives, which is the same number that
$\mathcal{M}(b)|_{\bar{k}}$ has.

A summand $M$ is said to \textit{start at $d$} if $d=\textup{min}\{i|\mathbb{Z}\{i\} \textup{ is a summand of }M_{\bar{k}}\}$. Similarly, a summand $M$ \textit{ends at $d$} if $d=\textup{max}\{i|\mathbb{Z}\{i\} \textup{ is a summand of }M_{\bar{k}}\}$. We will use the following theorem of Vishik. Here $i_W(q)$ denotes the Witt index of the quadratic form $q$. This is the number of hyperbolic plane summands in $q$.

\begin{thm}(\cite[4.15]{Vi04}) \label{Vish1}
Let $P,Q$ be smooth projective quadrics over $k$, and $d\geq 0$.
Assume that for every field extension $E/k$, we have that
$$i_W(p|_E)
>d \Leftrightarrow i_W(q|_E)> m.$$ Then the indecomposible summand
in $\mathcal{M}(P)$ starting at $d$ is isomorphic to the (Tate
twisted) indecomposible summand in $\mathcal{M}(Q)$ starting at
$m$.
\end{thm}

With this theorem, it becomes straight forward to prove the
following motivic decomposition (Thm.\ \ref{Q(J,u)motive}), by translating it into some
elementary facts about multiples of Pfister forms. First we will
state two lemmas for convenience.

\begin{lem} \label{Vish2}
Let $\phi$ be an $r$-fold Pfister form ($r \geq 1$) and let $b$ be an
$n$-dimensional non-degenerate quadratic form ($n \geq 2$). For any $0\leq d\leq \lfloor \frac{n}{2} \rfloor -1$,
we have $i_W(\phi \otimes b)>2^rd$ implies $i_W(\phi \otimes b)>2^r(d+1)-1$.
\end{lem}

\begin{proof}
This follows from the fact that if $\phi$ is anisotropic then $2^r$ divides $i_W(\phi \otimes b)$ \cite[Lemma 6.2]{Vi04} or \cite[Thm.\ 2(c)]{WS77}.
\end{proof}

\begin{lem} \label{Vish3}
If $Q$ is a smooth projective quadric of dimension $N$, then for any $0\leq d\leq N$, an indecomposable summand of $\mathcal{M}(Q)$ starting at $d$ is isomorphic (up to Tate twist) to an indecomposable summand of $\mathcal{M}(Q)$ ending at $N-d$. The same is true for indecomposable summands of $F^r_n$ for any $r\geq 1$ and $n\geq 1$.
\end{lem}

\begin{proof}
This is proved in \cite[Thm.\ 4.19]{Vi04} for anisotropic $Q$, but it is also true for isotropic $Q$ by using \cite[Prop.\ 2.1]{Vi04} to reduce to the anisotropic case. The statement for the motive $F^r_n$ follows easily from its construction.
\end{proof}

\begin{thm} \label{Q(J,u)motive}
Let $\phi$ be an $r$-fold Pfister form ($r\geq 1$), and for
non-zero $b_i$ and $n\geq 2$ we let $q=\phi \otimes \langle
b_1,\cdots,b_{n-1}\rangle \perp \langle b_n \rangle$ over $k$ of characteristic not $2$. Then we have
the following motivic decomposition. $$\mathcal{M}(q) = F^r_n
\oplus \bigoplus_{i=1}^{2^r-1}F^r_{n-1}\{i\} \oplus \left\{
\begin{array}{l l}
  \emptyset & \mbox{if $n$ is odd}\\
  \bigoplus_{j=1}^{2^{r-1}-1} R^r\{2^{r-1}(n-1)-j\} \hspace{2mm} & \mbox{if $n$ is even.}\\
\end{array} \right. $$
\end{thm}

\begin{proof}
We will split the proof into steps, including one step for each of the
three summands. We will use the notation $b'=\langle
b_1,\cdots,b_{n-1} \rangle$ and $b=b' \perp \langle b_n \rangle$. Note that we can assume that $\phi$ is anisotropic, because when it is isotropic both sides split into Tate motives, and we get the isomorphism by checking that on the right hand side there is exactly one copy of $\mathbb{Z}\{i\}$ for each $0\leq i \leq 2^r(n-1)$.

\hspace{4mm} \textit{Step 1: The first summand.} To show that $F^r_n$ is isomorphic to a summand of $\mathcal{M}(q)$, we need to show that given an indecomposable summand in $F^r_n$ starting at $d$, then there is an isomoprhic indecomposable summand in $\mathcal{M}(q)$ starting at $d$. In fact, by Lemma \ref{Vish3} it is enough to only consider indecomposable summands starting in the `first half', which is to say starting at $i\leq 2^{r-1}(n-1)$.

Since the only Tate motives in the first half of $F^r_n|_{\bar{k}}$ are $\mathbb{Z}[2^rd]$ for some $0\leq d \leq \lfloor \frac{n}{2} \rfloor -1$, by Thm.\ \ref{Vish1} it is enough to show that
for each such $d$ and $E/k$ field extension we have $i_W(\phi \otimes b|_E) >2^rd$ iff $i_W(\phi \otimes b'
\perp \langle b_n \rangle|_E)>2^rd$.

The ``if" part is clear. So assume
$i_W(\phi \otimes b|_E)>2^rd$. Then by Lemma \ref{Vish2} we
know $i_W(\phi \otimes b|_E)\geq 2^r(d+1)$. So the $2^r(d+1)$-dimensional
totally isotropic subspace must intersect the
$2^r-1$-codimensional subform $\phi \otimes b' \perp \langle b_n
\rangle \subset \phi \otimes b$ in dimension at least $2^rd+1$. In other words, $i_W(\phi \otimes b'
\perp \langle b_n \rangle|_E)>2^rd$.

\hspace{4mm} \textit{Step 2: The second summand.} Fix a $1\leq i \leq 2^r-1$.
As argued in Step 1, we want to show that if $0\leq d \leq \lfloor \frac{n}{2} \rfloor -1$, and if there is an indecomposable summand of $F^r_{n-1}$ starting at $2^rd$, then there is an isomorphic indecomposable summand of $\mathcal{M}(q)$ starting at $2^rd+i$.
By Thm.\ \ref{Vish1} it is enough to show that for any $E/k$ we have $i_W(\phi \otimes b'|_E) >2^rd$ iff $i_W(\phi \otimes b' \perp \langle b_n \rangle|_E)>2^rd+i$.
\begin{align*} i_W(\phi
\otimes b'|_E)>2^rd & \Rightarrow i_W(\phi \otimes b' \perp \langle b_n
\rangle|_E)>2^r(d+1)-1 & \quad \mbox{Lemma \ref{Vish2}} \\ & \Rightarrow
i_W(\phi \otimes b' \perp \langle b_n \rangle|_E)>2^rd+i & \\ &  \Rightarrow i_W(\phi \otimes
b')>2^rd & \quad \mbox{ See below}
\end{align*}

The last implication follows since the $\geq 2^rd+2$ dimensional
totally isotropic subspace must intersect the codimension 1
subform in dimension at least $2^rd+1$. So, by Lemma \ref{Vish3}, any indecomposable summand of $F^r_{n-1}$ has a corresponding summand in $\mathcal{M}(q)$, so we have shown that $F^r_{n-1}\{i\}$ is isomorphic to a summand of $\mathcal{M}(q)$.

\hspace{4mm} \textit{Step 3: The third summand.} Assume $n$ is even. Since the
summand is empty for $r=1$, we can assume $r\geq 2$. Fix an $2^{r-1}(n-2) < i < 2^{r-1}(n-1)$.
\begin{align*} i_W(\phi)>0 & \Rightarrow i_W(\phi)=2^{r-1} & \quad \mbox{Property of Pfister forms} \\
& \Rightarrow i_W(\phi \otimes b' \perp \langle b_n \rangle)> i & \\ &
\Rightarrow i_W(\phi)>0 & \quad \mbox{See below} \end{align*}

For the last implication, we have that the hyperbolic part of $\phi \otimes b' \perp \langle b_n \rangle$
is of dimension $\geq 2^r(n-2)+4$. So the anisotropic part is of dimension $\leq 2^r-2$. So by the Arason-Pfister hauptsatz, $\phi \otimes b'$ is hyperbolic. Now if $\phi$ were anisotropic, then $2\textup{dim}(\phi)$ would divide $\textup{dim} (\phi \otimes b')$ \cite[Thm.\ 2(c)]{WS77}. But this says $2^{r+1} | 2^r(n-1)$, which is impossible for $n$ even. Therefore $\phi$ is isotropic.

To finish Step 3, we use Thm.\ \ref{Vish1} to get the isomorphism
of motivic summands.

\hspace{4mm} \textit{Step 4: Counting Tate motives.} To finish the proof, one needs to show that the summands we
have described in these three steps are all possible
summands. This can easily be checked by counting the Tate motives
over $\bar{k}$. For a visualization of this, see Example \ref{MotiveExample} below. 

We have implicitly used \cite[Cor.\ 4.4]{Vi04} here. Note also that for the $n=2$ case the second summand is zero.
\end{proof}

\begin{ex} \label{MotiveExample}
As an illustration of the counting argument in Step 4 above, consider $r=2$ and $n=4$. Then Thm.\ \ref{Q(J,u)motive} says that $\mathcal{M}(\langle \langle a_1, a_2 \rangle \rangle \otimes \langle b_1, b_2, b_3 \rangle \perp \langle b_4 \rangle)$ has 5 motivic (possibly decomposible) summands in this decomposition. We can visualize this decomposition, as in \cite{Vi04}, with a node for each of the 12 Tate motives over $\bar{k}$, and a line between the nodes if they are in the same summand. Then the motive of the 11-dimensional quadric, $\mathcal{M}(q)$, is as follows, with each summand labelled:

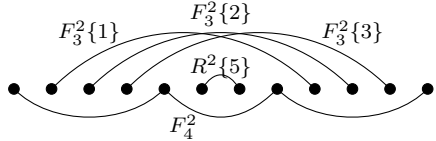
\hspace{30mm} \begin{tikzpicture}
\filldraw (0,0) circle (2pt)
(0.5,0) circle (2pt)
(1,0) circle (2pt)
(1.5,0) circle (2pt)
(2,0) circle (2pt)
(2.5,0) circle (2pt)
(3,0) circle (2pt)
(3.5,0) circle (2pt)
(4,0) circle (2pt)
(4.5,0) circle (2pt)
(5,0) circle (2pt)
(5.5,0) circle (2pt);
\draw (0.5,0) .. controls (1.5,1) and (3,1) .. (4,0);
\draw (1,0) .. controls (2,1) and (3.5,1).. (4.5,0);
\draw (1.5,0) .. controls (2.5,1) and (4,1).. (5,0);
\draw (0,0) .. controls (0.5,-0.5) and (1.5,-0.5) .. (2,0);
\draw (2,0) .. controls (2.5,-0.5) and (3,-0.5) .. (3.5,0);
\draw (3.5,0) .. controls (4,-0.5) and (5,-0.5) .. (5.5,0);
\draw (2.5,0) .. controls (2.66,0.25) and (2.84,0.25) .. (3,0);
\draw (2.75,0.3) node {\footnotesize $R^2\{5\}$};
\draw (2.25,-0.5) node {\footnotesize $F^2_4$};
\draw (1, 0.75) node {\footnotesize $F^2_3\{1\}$};
\draw (2.75, 1) node {\footnotesize $F^2_3\{2\}$};
\draw (4.5, 0.75) node {\footnotesize $F^2_3\{3\}$};
\end{tikzpicture}

\normalsize Notice that these summands might be decomposable, for example if the Pfister form $\langle \langle a_1, a_2 \rangle \rangle$ is split. So this differs slightly from Vishik's diagrams, since he used solid lines to denote indecomposable summands, and dotted lines for possibly decomposable ones.
\end{ex}

\stepcounter{thm}
\subsubsection{The motive of the base locus $Z_1$}

Now we will use our understanding of $Z_1$ from Thm.\ \ref{B1orbit} and its proof, to decompose its motive into the direct sum of Tate twisted Rost motives.

\begin{prop} \label{B1motive}
(1) For $r=1$, we have that $\mathcal{M}(Z_1, \mathbb{Z}/2)\cong
\oplus_{i=0}^{n-1} R^1\{i\}$ \\ (2) For $r=2$, we have that
$\mathcal{M}(Z_1, \mathbb{Z}/2) \cong \oplus_{i=0}^{2n-3}R^2\{i\}.$ \\ (3) For
$r=3$, we have that $\mathcal{M}(Z_1, \mathbb{Z}/2) \cong \oplus_{i=0}^7
R^3\{i\}.$
\end{prop}

\begin{proof}
For $r=1$, it is shown in \cite{Kr07} that $Z_1 \cong \mathbb{P}^{n-2} \times_k \textup{Spec}(k\sqrt{a_1})$.
We know that $\mathcal{M}(\textup{Spec}(k[\sqrt{a_1}]))\cong R^1$, so the result follows because the motive of projective space splits into Tate motives.

We have seen that in all cases $Z_1$ is a smooth scheme that is homogeneous for $\textup{Aut}(J,u)$.
Moreover, for $r=2$ or $3$, we know that $Z_1$ is a \textit{generically split} variety in the sense of \cite{PSZ08}.
So by their theorem \cite[5.17]{PSZ08} we have that $\mathcal{M}(Z_1,\mathbb{Z}/2)$ is isomorphic
to a direct sum of Tate twisted copies of an indecomposable motive $\mathcal{R}_2(\textup{Aut}(J,u))$.

Now let $V$ be the projective quadric defined by the $r$-Pfister form $\phi$, the norm
form of the composition algebra $C$.
It is a homogeneous $\textup{SO}(\phi)$ variety. Since $C$ splits over
the function field $k(V)$, by Jacobson's coordinatization theorem $J$ must also
split over $k(V)$, and therefore so does the group $\textup{Aut}(J,u)$. Furthermore,
over $k(Z_1)$, we have a rational point in $Z_1$. Then for any non-zero coordinate $c_i\in C$ of such
a point, there exists $0\neq y\in C$ such that $c_iy=0$ in $C$.
 But then $\phi(c_i)y=(\bar{c_i}c_i)y=\bar{c_i}(c_iy)=0$, and so $C$ has an isotropic vector, and is therefore split. Therefore $\textup{SO}(\phi)$ splits over $k(Z_1)$.

Now we may apply \cite[Prop.\ 5.18(iii)]{PSZ08} to conclude that $\mathcal{R}_2(\textup{Aut}(J,u)) \cong \mathcal{R}_2(\textup{SO}(\phi))$.
Finally, observe that $\mathcal{R}_2(\textup{SO}(\phi))$ is isomorphic to the Rost motive of $\phi$ (\cite[Last example in 7]{PSZ08}), which is the motive $R^r$.
The proposition can be deduced now by counting the Betti numbers of $Z_1$.
\end{proof}

\stepcounter{thm}
\subsubsection{Motivic decomposition of $X(J)$}

We are ready to decompose the motive $\mathcal{M}(X(J))$ for any
reduced simple Jordan algebra $J$. Recall that $X(J)$ is
a homogeneous space for $\textup{Aut}(J)$ (Lemma
\ref{Aut(J)-orbits}).

\begin{prop} \label{MotiveProp}
Let $r=0,1,2$ or $3$ and $n\geq 3$, and if $r=3$ then $n=3$. We have the following isomorphism of motives with coefficients in $\mathbb{Z}$.
$$\mathcal{M}(Q(J_n,u)) \oplus
(\bigoplus^{d_1-1}_{i=1}\mathcal{M}(Z_1)\{i\}) \cong
\mathcal{M}(X(J_n)) \oplus (\bigoplus_{i=1}^{d_2-1}
\mathcal{M}(X(J_{n-1}))\{i\}).$$ Here $d_i$ are the respective
codimensions of the subschemes $B_i$. In particular, for $r\neq 0$, $d_1=2^{r-1}n-2$ and $d_2=2^r$.
\end{prop}

\begin{proof}
If $n$ is the degree of $J_n$, we have by Section \ref{BlowupSect}
that the blow up of $X(J_n)$ along the smooth subvariety
$X(J_{n-1})$ is isomorphic to the blow up of $Q(J_n,u)$ along the
smooth subscheme $Z_1$. So by applying the blow up formula for
motives \cite[p.463]{Ma68}, we get the above isomorphism.
\end{proof}

\begin{thm} \label{MotiveX(J)}
Let $r=0,1,2$ or $3$, and $n\geq 3$ (and if $r=3$ then $n=3$). And let $J=\textup{Sym}(M_n(C),\sigma_b)$ where $C$ is
a $2^r$-dimensional composition algebra over $k$, and $b=\langle
b_1,\cdots,b_n \rangle$ is a non-degenerate quadratic form over
$k$. Then \\
$(r=0):$ $$\mathcal{M}(X(J)) \cong F^0_n = \mathcal{M}(b),$$ \\
$(r=1):$ $$\mathcal{M}(X(J), \mathbb{Z}/2) \cong F^1_n \oplus
\bigoplus_{j=0}^{\lfloor \frac{n-3}{2} \rfloor} \left(
\bigoplus_{i=1}^{2\lfloor \frac{n}{2} \rfloor} R^1\{i+2j\}
\right),$$ \\
$(r=2):$ $$\mathcal{M}(X(J), \mathbb{Z}/2) \cong F^2_n \oplus
\bigoplus_{j=0}^{\lfloor \frac{n-2}{2} \rfloor} \left(
\bigoplus_{i=1}^{4\lfloor \frac{n-1}{2} \rfloor+1} R^2\{i+4j\}
\right),$$ \\
$(r=3):$ $$\mathcal{M}(X(J), \mathbb{Z}/2) \cong F^3_3 \oplus \bigoplus_{i=1}^{11}
R^3\{i\}.$$
\end{thm}

\begin{proof}
The motive of $Q(J,u)$ may be decomposed in terms of the motives
$F^r_n$, $F^r_{n-1}$ and $R^r$ (Thm.\ \ref{Q(J,u)motive}). The
motive of $Z_1$ with $\mathbb{Z}/2$ coefficients may be decomposed in terms of $R^r$ (Prop.\
\ref{B1motive}). The subvariety $X(J_2)$ is
isomorphic to the quadric defined by $\phi \otimes \langle b_1 \rangle \perp \langle b_2 \rangle$ (see proof of Lemma \ref{B2smooth}), so we have already decomposed its motive in terms
of $F^r_2$ and $R^r$ (Thm.\ \ref{Q(J,u)motive}).

So the last ingredient we need is the cancellation theorem. It gives conditions for when it
is true that an isomorphism of motives $A \oplus B \cong A \oplus C$ implies an isomorphism
of motives $B \cong C$. 
This does not hold in general; there are counter-examples when $\Lambda = \mathbb{Z}$ \cite[Remark 2.8]{CPSZ06}. But if we take $\Lambda$ to be any field, then the stronger Krull-Schmidt theorem holds, which says that any motivic decomposition into indecomposables is unique \cite[Thm.\ 34]{CM06}\footnote{Although this theorem is only stated for $\Lambda$ a discrete valuation ring, the same proof works for any field.}.

When we put these pieces into the isomorphism from Prop.\ \ref{MotiveProp}, we may proceed by induction on $n$. One sees that we can cancel  the $F^r_{n-1}$ terms in the decomposition, leaving us with the motive $\mathcal{M}(X(J))$ on the right hand side, $F^r_n$ on the left hand side, and  several Tate twisted copies of $R^r$ on both sides. To finish the proof one just needs to count the number of copies of $R^r$ remaining after the cancellation theorem, and verify that the given expressions are correct. We leave this induction argument to the reader. 
\end{proof}

\begin{remark}
When $\phi$ is isotropic, the above motives split. When $\phi$ is anisotropic, $R^r$ is indecomposable, but the motive $F^r_n$ could still be decomposable, depending on the quadratic form $b$.
\end{remark}

\begin{remark} \label{Krashen}
The $r=1$ case of the above theorem may be used to prove Krashen's motivic equivalence \cite[Thm.\ 3.3]{Kr07}. To see this, notice that a $1-$Pfister form $\phi$ defines a quadratic \'etale extension $l/k$, and any hermitian form $h$ over $l/k$ is defined by a quadratic form $b$ over $k$. So in Krashen's notation, $V(h)=X(J)$. Furthermore, his $V(q_h)$ is the projective quadric defined by $\phi \otimes b$, and his $\mathbb{P}_L(N)$ is isomorphic to the base locus $Z_1$. So in the notation of this paper, his motivic equivalence is $$\mathcal{M}(\phi \otimes b) \oplus \bigoplus_{i=1}^{n-2} \mathcal{M}(Z_1)\{i\} \cong \mathcal{M} (X(J)) \oplus \mathcal{M}(X(J)) \{1\}.$$

Since we have motivic decompositions of all of these summands in terms of $F^1_n$ and $R^1$ (see Thm.\ \ref{Motiveq_h}, Prop.\ \ref{B1motive} and Thm.\ \ref{MotiveX(J)}), it is easy to verify his motivic equivalence, at least for $\mathbb{Z}/2$ coefficients.

On the other hand, the $r=1$ case of Thm.\ \ref{MotiveX(J)} follows from Krashen's motivic equivalence, together with the $r=1$ cases of Thm.\ \ref{Motiveq_h} and Prop.\ \ref{B1motive}; this is pointed out in \cite[Thm.\ (C)]{SZ08}.
\end{remark}

\textbf{Acknowledgements.} \hspace{1mm} I would like to thank Burt Totaro for many useful discussions, and for originally noticing the $r=1$, $n=3$ case of Prop.\ \ref{TranspositionMaps}. I would also like to thank Patrick Brosnan, Nikita Semenov and the referee for their comments.

\small Mark L. MacDonald. Department of Mathematics, University of British Columbia, Vancouver, B.C. Canada V6T 1Z2. \texttt{mlm@math.ubc.ca}


\begin{thebibliography}{50}
\small

\bibitem[AKMW02]{AKMW02} \textsc{D. Abramovich}, \textsc{K. Karu}, \textsc{K. Matsuki} and \textsc{J. Wlodarczyk}. Torification and factorization of birational maps. \textit{J. Amer. Math. Soc.} \textbf{15} (2002), 531-572.

\bibitem[Ah86]{Ah86} \textsc{D. Ahiezer}. Algebraic groups
transitive in the complement of a homogeneous hypersurface.
\textit{Trans. Moscow Math. Soc.} \textbf{48} (1986), 83-103.

\bibitem[Al47]{Al47} \textsc{A.A. Albert}. A structure theory for
Jordan algebras. \textit{Ann. of Math. (2)} \textbf{48} (1947),
546-567.

\bibitem[Bo91]{Bo91} \textsc{A. Borel}. \textit{Linear algebraic
groups}. Graduate Texts in Mathematics (Springer, 1991).

\bibitem[CPSZ06]{CPSZ06} \textsc{B. Calm\`es}, \textsc{V. Petrov}, \textsc{N. Semenov} and \textsc{K. Zainoulline}. Chow motives of twisted flag varieties. \textit{Compositio Math.} \textbf{142} (2006), 1063-1080.

\bibitem[Ch06]{Ch06} \textsc{P.-E. Chaput}. Geometry over composition algebras: projective geometry. \textit{J. Algebra} \textbf{298} (2006), 340-362.

\bibitem[CM06]{CM06} \textsc{V. Chernousov} and \textsc{A. Merkurjev}. Motivic decomposition of projective homogeneous varieties and the Krull-Schmidt theorem. \textit{Transform.\ Groups} \textbf{11} (2006), 371-386.

\bibitem[Co94]{Co94} \textsc{A. Corti}. Factoring birational maps of threefolds after Sarkisov. \textit{J. Algebraic Geom.} \textbf{4} (1994), 223-254.

\bibitem[De01]{De01} \textsc{O. Debarre}. \textit{Higher-dimensional
algebraic geometry}. Universitext (Springer, 2001).

\bibitem[EKM08]{EKM08} \textsc{R. Elman}, \textsc{N. Karpenko} and \textsc{A. Merkurjev}. \textit{The algebraic and geometric theory of quadratic forms}. Colloquium Publications (American Mathematical Society, 2008).

\bibitem[EH00]{EH00} \textsc{D. Eisenbud} and \textsc{J. Harris}.
\textit{The geometry of schemes}. Graduate Texts in Mathematics
(Springer, 2000).

\bibitem[FH91]{FH91} \textsc{W. Fulton} and \textsc{J. Harris}. \textit{Representation theory: A first course}. Graduate Texts in Mathematics (Springer, 1991).

\bibitem[Ha77]{Ha77} \textsc{R. Hartshorne}. \textit{Algebraic
geometry}. Graduate Texts in Mathematics (Springer, 1977).

\bibitem[Ja68]{Ja68} \textsc{N. Jacobson}. \textit{Structure and
representations of Jordan algebras}. Amer. Math. Soc. Colloq.
Publ. vol. XXXIX (American Mathematical Society, 1968).

\bibitem[Ja85]{Ja85} \textsc{N. Jacobson}. Some projective
varieties defined by Jordan algebras. \textit{J. Algebra}
\textbf{97} (1985), 565-598.

\bibitem[KMRT98]{KMRT98} \textsc{M.-A. Knus}, \textsc{A. Merkurjev}, \textsc{M. Rost} and \textsc{J.-P.
Tignol}. \textit{The book of involutions}. Colloquium Publications
(American Mathematical Society, 1998).

\bibitem[Kr07]{Kr07} \textsc{D. Krashen}. Motives of unitary and
orthogonal homogeneous varieties. \textit{J. Algebra} \textbf{318}
(2007), 135-139.

\bibitem[Mac08]{Mac08} \textsc{M.L. MacDonald}. Cohomological
invariants of odd degree Jordan algebras. \textit{Math. Proc.
Cambridge Philos. Soc.} \textbf{145} (2008), no.2, 295-303.

\bibitem[Ma68]{Ma68} \textsc{Ju. I. Manin}. Correspondences,
motifs and monoidal transformations. \textit{Math. USSR-Sb.}
\textbf{6} (1968), 439-470.

\bibitem[Mc04]{Mc04} \textsc{K. McCrimmon}. \textit{A taste of Jordan
algebras}. Universitext (Springer, 2004).

\bibitem[PSZ08]{PSZ08} \textsc{V. Petrov}, \textsc{N. Semenov} and
\textsc{K. Zainoulline}. $J$-invariant of linear algebraic groups. \textit{Ann. Sci. Ecole Norm. Sup.} (4) \textbf{41} (2008), no.6, 1023-1053. 

\bibitem[Ro05]{Ro05} \textsc{S. Roussey}. Isotropie, corps de fonctions et \'equivalences birationelles des formes quadratiques. Ph.D.\ thesis, Universit\'e de Franche-Comt\'e (2005).

\bibitem[SZ08]{SZ08} \textsc{N. Semenov} and \textsc{K. Zainoulline}. Essential dimension of Hermitian spaces. \\ Preprint, \texttt{arxiv.org/abs/0805.1161}

\bibitem[Se02]{Se02} \textsc{J.-P. Serre}. \textit{Galois Cohomology}.
Graduate Texts in Mathematics (Springer, 2002).

\bibitem[SV00]{SV00} \textsc{T.A. Springer} and \textsc{F.D. Veldkamp}. \textit{Octonions,
Jordan algebras and exceptional groups}. Springer Monographs in
Mathematics (Springer, 2000).

\bibitem[Ti65]{Ti65} \textsc{J. Tits}. Classification of algebraic semisimple
groups. \textit{Algebraic Groups and Discontinuous Subgroups}.
Proc. Sympos. Pure Math. IX (Amer. Math. Soc., Providence, R.I.,
1966), 32-62.

\bibitem[To08]{To08} \textsc{B. Totaro}. Birational geometry of quadrics. \textit{Bull.\ Soc.\ Math.\ France}, to appear.

\bibitem[Vi04]{Vi04} \textsc{A. Vishik}. Motives of quadrics with applications to the theory of quadratic
forms. \textit{Geometric Methods in the Algebraic Theory of
Quadratic Forms. Summer School, Lens 2000} (ed. J.-P. Tignol),
Lect. Notes in Math., 1835 (Springer, Berlin, 2004), 25-101.

\bibitem[WS77]{WS77} \textsc{A.R. Wadsworth} and \textsc{D.B. Shapiro}. On multiples of round and
Pfister forms. \textit{Math.\ Z.} \textbf{157} (1977), 53-62.

\bibitem[Za93]{Za93} \textsc{F.L. Zak}. \textit{Tangents and secants of algebraic varieties}. Translations of Math. Monographs, 127 (Amer. Math. Soc., Providence, R.I., 1993).

\end{thebibliography}
\end{document}